\documentclass[letterpaper, 10 pt, conference]{ieeeconf}

\IEEEoverridecommandlockouts
\usepackage{amsmath,amssymb,amsfonts,mathtools,bm}
\usepackage[ruled,vlined]{algorithm2e}
\usepackage{tikz}
\usetikzlibrary{arrows.meta, positioning, calc}

\DeclareMathOperator*{\argmin}{arg\,min}
\newcommand{\Prob}{\mathbb P}
\newcommand{\E}{\mathbb E}

\newcommand{\F}{\mathcal F}

\newcommand{\R}{\mathbb R}
\newcommand{\Ss}{\mathbb S}
\newcommand{\diag}{\mathrm{diag}}

\newcommand{\transpose}{^\mathsf{T}}

\usepackage{amsthm}

\newtheorem{proposition}{Proposition}
\newtheorem{remark}{Remark}
\newtheorem{assumption}{Assumption}

\usepackage{color}
\usepackage{xcolor}
\usepackage{cite}
\usepackage[colorlinks=true, allcolors=blue]{hyperref}

\title{\LARGE \bf
Information Revelation and Alignment Faking in Stochastic Differential Games}

\author{Daniel Ralston$^{1}$, Xu Yang$^{1}$ and Ruimeng Hu$^{1,2}$
\thanks{This work was partially supported by the ONR grant \#N00014-24-1-2432, the Simons Foundation (MP-TSM-00002783) and the NSF grant DMS-2420988.}
\thanks{$^{1}$Department of Mathematics, University of California, Santa Barbara, CA 93106, USA.
Email: \texttt{\{danielralston, xy6, rhu\}@ucsb.edu}.}
\thanks{$^{2}$Department of Statistics and Applied Probability, University of California, Santa Barbara, CA 93106, USA.}
}

\begin{document}

\maketitle
\thispagestyle{empty}
\pagestyle{empty}

\begin{abstract}
In competitive games with private objectives, actions can reveal information about hidden parameters. However, quantifying such information revelation is nontrivial, since it depends not only on the opponent's hidden parameter but also on the opponent's model of the game. We study this problem via a two-player linear-quadratic stochastic differential game under partial information, in which each player knows its own coupling parameter and models the opponent’s hidden parameter through a prior. Starting from the full-information game, we characterize the Nash equilibrium by coupled Riccati equations. We then define baseline implementable controls by averaging the equilibrium under each player’s prior.  Building on this baseline, we formulate an alignment faking control problem in which one player trades off fidelity to its implementable policy against information acquisition about the opponent’s hidden parameter. The information incentive is constructed from a proxy Fisher information matrix based only on the player’s available model. This leads to a tractable saddle-point formulation with semi-explicit control characterization through Riccati systems. Numerical illustrations show that alignment faking can substantially improve information gain over baseline play when the faker’s model is accurate, but often at the cost of greater detectability. They also show that the proxy Fisher information can systematically differ from the true information under model misspecification.
\end{abstract}


\section{Introduction}
Information can be revealed by actions in competitive games with private objectives. By observing an opponent's behavior, a player may infer features of its hidden objective. However, quantifying such information revelation is nontrivial, since it depends not only on the opponent's private objective but also on the opponent's model of the game, a challenge described by Townsend in \cite{townsend1983forecasting} as ``forecasting the forecasts of others''. In games with private information, a player must therefore reason through a proxy for the true information, constructed from its own available quantities. 


This raises a further strategic question: Can a player deliberately modify its behavior to reveal more information about an opponent's hidden objective, while still appearing close to ordinary play? This paper addresses that question through what we call \emph{alignment faking (AF) control}. In our setting, AF refers to a control strategy that remains close to the player's baseline implementable control while intentionally steering the trajectory to extract more information about the opponent's hidden parameter.

The term ``alignment faking'' is borrowed from the AI safety literature, where it refers to systems that strategically conceal their true objectives from overseers \cite{greenblatt2024alignment,hubinger2019risks}. Our model gives a continuous-time  game-theoretic analogue: the player using the AF control plays the role of the AI system, while the opponent plays the role of the overseer. Unlike much of the AI safety literature, where the objective of AF is often left implicit, our framework makes the underlying tradeoff explicit through a control objective which balances information gain and closeness to baseline behavior.

To study this tradeoff, we consider a two-player linear-quadratic stochastic differential game; see Figure~\ref{fig:schematic} for a schematic illustration. Within this setting, we develop an analytically tractable quantitative framework that addresses how private beliefs shape information revelation, how purposeful deviations from baseline behavior affect information gain, and how such deviations may be detected. The main contributions of the paper are summarized below.



\begin{figure}[t h]
    \centering
    \resizebox{\columnwidth}{!}{
        \begin{tikzpicture}[
            >=Stealth,
            playerbox/.style={draw, rounded corners=3pt, minimum width=3.4cm, minimum height=1.6cm, align=center, font=\normalsize},
            controlbox/.style={draw, rounded corners=2pt, minimum width=1.5cm, minimum height=0.75cm, align=center, font=\normalsize},
            trajbox/.style={draw, thick, rounded corners=3pt, minimum width=5cm, minimum height=0.7cm, align=center, font=\normalsize},
            questionbox/.style={draw, thick, rounded corners=5pt, minimum width=7cm, align=left, font=\normalsize, inner sep=6pt},
            bluearrow/.style={->, thick, blue!70!black},
            redarrow/.style={->, thick, red!70!black},
        ]
         
        \node[playerbox, fill=blue!8] (A) {
            \textbf{Player $A$ (AF)}\\[2pt]
            Public state: $X_t^A$, $t \in [0,T]$\\
            Hidden private parameter: $m_A$\\
            Goal: track $m_A X_t^B$\\
            Internal model: prior $\pi^B$ for $m_B$ 
        };
         
        \node[playerbox, fill=red!8, right=1.cm of A] (B) {
            \textbf{Player $B$ (Opponent)}\\[2pt]
            Public state: $X_t^B$, $t\in [0,T]$\\
            Hidden private parameter: $m_B$\\
            Goal: track $m_B X_t^A$ \\
            Internal model: prior $\pi^A$ for $m_A$ 
        };
         
        \node[controlbox, fill=blue!8, below left=0.7cm and -2.cm of A] (AF) {AF control\\$v^\ast$};
        \node[controlbox, fill=blue!8, below right=0.6cm and -2.cm of A] (baseA) {Baseline\\ implementable control \\ $\overline{u}^{A,\ast}$};
         
        \node[controlbox, fill=red!10, below=0.6cm of B] (baseB) {Baseline\\ implementable control \\$\widetilde{u}^{B,\ast}$};
         
        \draw[bluearrow] (A.south) -- ++(0,-0.15) -| (AF.north);
        \draw[bluearrow] (A.south) -- ++(0,-0.15) -| (baseA.north);
        \draw[redarrow] (B.south) -- (baseB.north);
         
        \node[trajbox, below=2.5cm of $(A.south)!0.5!(B.south)$] (traj) {
            Observed trajectories \;$X_t^A,\; X_t^B$
        };
         
        \draw[bluearrow] (AF.south) -- ++(0,-0.4) -| ([xshift=-0.6cm]traj.north);
        \draw[bluearrow] (baseA.south) -- ++(0,-0.1) -| ([xshift=0.0cm]traj.north);
        \draw[redarrow] (baseB.south) -- ++(0,-0.2) -| ([xshift=0.6cm]traj.north);
         
        \node[font=\small, text=black] at ($(AF.east)!0.5!(baseA.west)$) {or};
         
        \node[questionbox, below=0.5cm of traj] (Q) {
            \begin{minipage}{6.6cm}
            \begin{enumerate}\setlength{\itemsep}{1pt}\setlength{\parskip}{0pt}
                \item[\textbf{Q1.}] How does player $A$'s information about $m_B$ change depending on $\pi^A$ and $\pi^B$?
                \item[\textbf{Q2.}] Can the AF control yield trajectories that reveal more information about $m_B$?
                \item[\textbf{Q3.}] Can player $B$ detect the use of alignment faking $v^\ast$ from observed trajectories?
            \end{enumerate}
            \end{minipage}
        };
         
        \draw[->, thick, black] (traj.south) -- (Q.north);
         
        \end{tikzpicture}
    }
    \caption{Schematic of the alignment faking (AF) game. Each player knows its own coupling parameter and holds a prior on the opponent's. Player $A$ may deviate from the baseline implementable control via AF, while player $B$ plays the baseline.
    }
    \label{fig:schematic}
\end{figure}
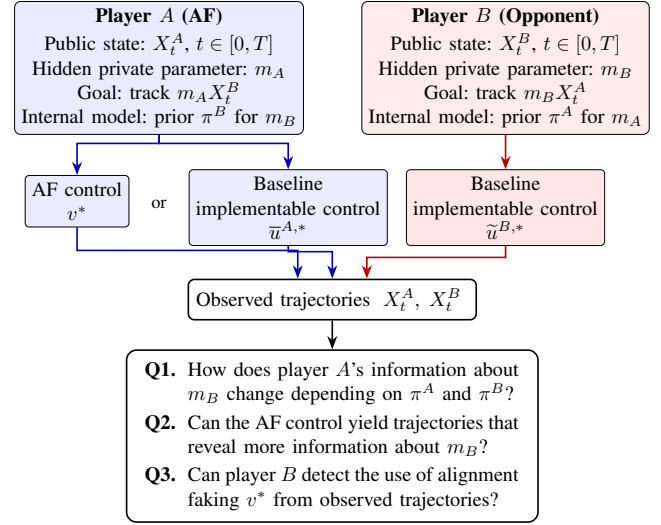


\noindent \textbf{1. Symmetric partial-information game with alignment faking:} We formulate a two-player continuous-time stochastic differential game under partial information, extending earlier Stackelberg-type models \cite{hu2025strategic} to a symmetric setting. We derive implementable controls as expectations of the full-information Nash equilibrium under each player's prior on the opponent's hidden parameter, which provides a tractable and consistent surrogate for equilibrium play in the partial-information game.  We then introduce an additional information-seeking control objective defined through a proxy Fisher information matrix constructed from the player's available quantities. We obtain semi-explicit characterizations of the implementable and AF controls, and establish their well-posedness through coupled Riccati systems under explicit time horizon bounds (Propositions~\ref{prop:theta_bound}--\ref{prop:exist_unique_thetaAF}).

\noindent\textbf{2. Numerical validation and discovery:} Numerical experiments show that the AF control can yield substantial information gains over baseline play when the alignment faker's model is well specified, but at the cost of greater detectability. We also show that the proxy Fisher information can systematically misestimate the true Fisher information under model misspecification. Finally, we find that information quality depends mainly on the alignment faker's own model, while the opponent's model has a secondary but noticeable effect.


\noindent {\it Related Literature}. Dynamic games have long been used to model strategic interactions among competing agents \cite{bacsar1998dynamic, carmona2016lectures, zhou2025adversarial,zhou2025integrating,kim2026stealthy,kim2024defining}. Games with incomplete information, where parameters in the dynamics or cost functionals are not known to all players, are especially relevant in adversarial settings arising in areas such as economics \cite{morris2002social}, cybersecurity \cite{zhang2020game}, and, more recently, AI systems \cite{hubinger2019risks}. Previous works have studied Nash equilibria in linear-quadratic games with asymmetric information using the stochastic maximum principle \cite{chang2014linear} and with continuous type spaces via viscosity solutions of Hamilton--Jacobi equations \cite{cardaliaguet2012games}. 
In incomplete-information games where the players share a common prior, a Bayesian Nash equilibrium \cite{harsanyi1967games} can be found by augmenting the state with the players' posterior beliefs over the unknown parameters and solving for a Nash equilibrium on this belief-augmented space.
Because our game features hidden parameters without common priors, and each player's prior is itself unknown to the opponent, we define implementable controls by averaging the full-information Nash equilibrium under each player's prior, and then study how a player can strategically deviate from this baseline to improve information extraction. 

Our work is also related to inverse learning and inverse optimal control, where the goal is to infer cost functions or unknown parameters from observed behavior \cite{molloy2017inverse,peters2021inferring,hu2025strategic,ward2025active}. In the linear-quadratic setting, this inverse problem is closely tied to the coupled Riccati equations governing optimal controls, whose existence and solvability have been widely studied \cite{abou2012matrix,papavassilopoulos1979existence}. The key difference is that we do not study passive inference alone. Instead, we introduce an active control problem in which one player deliberately modifies its behavior to increase the information revealed by the trajectory.

\noindent {\it Notations.}
Boldface letters denote vectors (e.g. $\bm{X}_t$, $\bm{W}_t$, $\bm{z}$), and normal typeface denotes scalar quantities. 
Random variables modeling unknown parameters are capitalized (e.g. $M_A$), while corresponding known or realized values are lowercase (e.g. $m_A$).
Super/subscripts $A$ (resp. $B$) index values associated with player $A$ (resp. player $B$), and the superscript $AF$ indexes values associated with player $A$'s alignment faking problem.  
The operator $\diag\{\cdot\}$ constructs a (block) diagonal matrix from its arguments.  We use $(\cdot)\transpose$ for transpose, $\|\cdot\|_F$ for the Frobenius norm, and $\mathbb{S}^{n\times n}$ for the space of $n\times n$ real symmetric matrices.  An overline $\overline{(\cdot)}$ and a tilde $\widetilde{(\cdot)}$ denote quantities computed using player $A$'s and player $B$'s available information (prior), while $\overline{\widetilde{(\cdot)}}$ denotes a player $A$ proxy for a player $B$ quantity. 

\section{Problem Setup and Baseline Nash Equilibrium}\label{sec:setup}

This section formulates a stochastic differential game between two competing players, \(A\) and \(B\). Assuming all model parameters are known to both players, we characterize the Nash equilibrium through a coupled Riccati system and provide conditions for its global existence. This setting serves as a baseline for the partial-information case.

\subsection{State Dynamics and Cost Functionals}
Fix $T>0$ and let $(\Omega,\F, \{\F_t\}_{t\geq 0}, \Prob)$ be a filtered probability space supporting a two-dimensional Brownian motion $\bm{W}_t = [ W_t^A, W_t^B] \transpose$ with independent components.
We model players $A$ and $B$ by the $\R^2$-valued state process  $\bm{X}_t = [X_t^A, X_t^B]\transpose$, controlled by $\bm{u}_t = [ u_t^A, u_t^B]\transpose$ and subject to noise from $\bm{W}_t$, evolving according to
\begin{equation}\label{eq:state_dynamics}
   dX_t^i = u_t^i \, dt + \sigma_i \, dW_t^i, \quad  i \in \{A, B\}, \quad \sigma_i > 0. 
\end{equation}

Let $m_A, m_B \in \R$ be fixed parameters. Player $i$ minimizes the cost functional 
\begin{equation}
    \label{eq:cost_A}
    J^i[\bm u] := \E\!\Big[\int_0^T f^i(t,\bm{X}_t,\bm{u}_t;m_i)\, dt\Big], \quad i \in \{A,B\},
\end{equation}
where the running cost of player $i$ is quadratic, defined as 
\begin{equation*}
    f^i(t,\bm{x},\bm{u};m_i) := q_i(x^i - m_i x^{-i})^2 + r_i (u^i)^2,
\end{equation*}
with $q_i, r_i > 0$, and $-i$ refers to the player who is not $i$.  
The running cost is interpreted as a tracking objective: player $i$ steers its state toward a fixed  multiple $m_i$ of the opponent's state.

For the baseline full-information game, we assume both players know both $m_A$ and $m_B$, the dynamics \eqref{eq:state_dynamics}, and cost functionals \eqref{eq:cost_A}.  We restrict the controls to be in the admissible sets, $i \in \{A,B\}$,
\begin{equation*}
    \mathcal{U}^{i} := \Big\{u^{i} \, \Big| \, \E\int_0^T |u_s^i|^2  ds < \infty, \bm{u}_t \in \sigma(\bm{X}_t)\Big\}.
\end{equation*}
The notion of interest is a Nash equilibrium (NE), a pair $(u^{A, \ast}, u^{B, \ast}) \in \mathcal{U}^A \times \mathcal{U}^B$ such that, for all $u^A \in \mathcal{U}^A$, $u^B \in \mathcal{U}^B$,
\begin{align*}
    J^{A}[u^{A, \ast}, u^{B,\ast}] &\leq J^{A}[u^{A}, u^{B,\ast}], \\
    J^{B}[u^{A, \ast}, u^{B,\ast}] &\leq J^{B}[u^{A,\ast}, u^{B}].
\end{align*}
In particular, we focus on Markovian NE of the form $u^{i, \ast}_t = u^{i, \ast}(t, \bm x)$, where $u^{i, \ast}: [0,T] \times \R^2 \to \R$ is measurable, \(i\in\{A,B\}\).



\subsection{Baseline Nash Equilibrium}\label{subsec:baseline_NE}

\noindent For $i \in \{A,B\}$, define the  best-response value function
\begin{equation*}
    V^i(t,\bm{x}) := \inf\limits_{u^i\in\mathcal{U}^i} \E\Big[\int_t^T f^i(s,\bm{X}_s,\bm{u}_s;m_i)\,ds \mid \bm X_t = \bm x\Big].
\end{equation*}
By dynamic programming \cite[Section~5.1.4]{carmona2016lectures}, $V^A$ and $V^B$ satisfy a coupled Hamilton-Jacobi-Bellman (HJB) system 
\begin{align*}
    0 = \partial_t V^A + \inf_{u^A}\big\{u^A \partial_{x^A} V^A + f^A\big\} + u^{B,\ast} \partial_{x^B} V^A 
    \\ + \tfrac{\sigma_A^2}{2}\partial^2_{x^A x^A}V^A + \tfrac{\sigma_B^2}{2}\partial_{x^B x^B}^2 V^A,\\
    0 = \partial_t V^B + \inf_{u^B}\big\{u^B \partial_{x^B} V^B + f^B\big\} + u^{A,\ast} \partial_{x^A} V^B \\ + \tfrac{\sigma_A^2}{2}\partial^2_{x^A x^A}V^B + \tfrac{\sigma_B^2}{2}\partial_{x^B x^B}^2 V^B,
\end{align*}
with terminal conditions $V^i(T, \bm{x}) = 0$, where $u^{i,\ast} = \argmin\limits_{u^i\in \R}\{u^i \partial_{x^i} V^i + f^i\}$.
Using the quadratic ansatz $V^i(t,\bm{x}) = \bm{x}\transpose \theta^i(t) \, \bm{x}  
+ \psi^i(t)$, the NE takes the form
\begin{align}
    u^{A,*}(t,\bm{x}; m_A, m_B) &= -\tfrac{1}{r_A}\!\left(x^A \theta_{11}^A(t) + x^B \theta_{12}^A(t) 
    \right), \label{eq:baseline_NE_control_A} \\
    u^{B,*}(t,\bm{x}; m_A, m_B) &= -\tfrac{1}{r_B}\!\left(x^A \theta_{12}^B(t) + x^B \theta_{22}^B(t) 
    \right). \label{eq:baseline_NE_control_B}
\end{align}
Here, for each pair $(m_A, m_B)$, the matrices 
 $\theta^A(t)$, $\theta^B(t) \in \Ss^{2\times 2}$  solve the coupled Riccati system 
\begin{align}\label{eq:Riccati_A}
    \dot{\theta}^A(t) &= \theta^A R_A \theta^A + \theta^A R_B \theta^B + \theta^B R_B \theta^A - Q_A, \\
    \label{eq:Riccati_B}
    \dot{\theta}^B(t) &= \theta^B R_B \theta^B + \theta^A R_A \theta^B + \theta^B R_A \theta^A - Q_B,
\end{align}
with terminal conditions $\theta^A(T) = \theta^B(T) = 0$, where \(\theta_{jk}^i(t)\) denotes the \((j,k)\)-entry of \(\theta^i(t)\), and 
\begin{align*}
    Q_A &= \begin{bmatrix} q_A & -m_A q_A \\ -m_A q_A & m_A^2 q_A \end{bmatrix}, &
    R_A &= \begin{bmatrix} \tfrac{1}{r_A} & 0 \\ 0 & 0 \end{bmatrix}, \\
    Q_B &= \begin{bmatrix} m_B^2 q_B & -m_B q_B \\ -m_B q_B & q_B \end{bmatrix}, &
    R_B &= \begin{bmatrix} 0 & 0 \\ 0 & \tfrac{1}{r_B} \end{bmatrix}.
\end{align*}
These equations follow by substituting the ansatz into the coupled HJB system and matching coefficients. Later, we sometimes  put $(m_A, m_B)$ after the semicolon and write $u^{i,*}(t,\bm{x}; m_A, m_B)$ and $\theta^i(t; m_A, m_B)$, to emphasize the dependence on both parameters. 


Next, we state a proposition that will be important for the existence of implementable controls, discussed in Section~\ref{sec:partial_info}.


\begin{proposition}\label{prop:theta_bound}
    Let $\theta = \diag\{\theta^A, \theta^B\}$, and suppose \vspace{-0.3em}    
    \begin{multline}\label{eq:T_bound}
        T \leq \mathcal{T}(m_A, m_B)
        := (c_1 c_2)^{-1/2}  \\ \cdot\arctan \big((1+m_A^2 + m_B^2)(c_2/c_1)^{-1/2}\big),
    \end{multline}
    where $c_1 = 5 (r_A^{-2} + r_B^{-2})^{1/2}$ and $c_2= \sqrt{q_A^2 + q_B^2}\,(1 + m_A^2 + m_B^2).$ 
    Then any solution to \eqref{eq:Riccati_A}--\eqref{eq:Riccati_B} on $[0,T]$ satisfies $\big\|\theta(t; m_A, m_B)\big\|_F \le 1 + m_A^2 + m_B^2$. 
\end{proposition}

\begin{proof}
Write the coupled Riccati system in block form as
$\dot{\theta} = F(\theta) - Q$ where $Q := \diag\{Q_A, Q_B\}$, 
    \begin{multline*}
        F(\theta) := \theta \begin{bmatrix} R_A & 0 \\ 0 & R_B \end{bmatrix} \theta + \theta \begin{bmatrix} 0 & R_B \\ R_A & 0 \end{bmatrix}\theta \begin{bmatrix} 0 & I_2 \\ I_2 & 0 \end{bmatrix} \\
        + \Big(\theta \begin{bmatrix} 0 & R_B \\ R_A & 0 \end{bmatrix}\theta \begin{bmatrix} 0 & I_2 \\ I_2 & 0 \end{bmatrix}\Big)\transpose. 
    \end{multline*}
    By submultiplicativity of the Frobenius norm, one has $\|F(\theta)\|_F \le c_1 \|\theta\|_F^2, \; \|Q\|_F \le c_2$. 
    Now define the time-reversed variable \(Y(t):=\theta(T-t)\). Then $
\dot{Y}(t) = -F(Y(t))+Q$,
and hence
$
\|\dot{Y}(t)\|_F
\le \|F(Y(t))\|_F+\|Q\|_F
\le c_1\|Y(t)\|_F^2+c_2$.

Consider the scalar ODE $\dot{u} = c_1 u^2 + c_2$, $u(0) = 0$, with explicit solution $u(t) = \sqrt{c_2/c_1}\,\tan(\sqrt{c_1 c_2}\, t)$. It remains finite for $t < \pi/(2\sqrt{c_1 c_2})$.  Restricting $t \leq \mathcal{T}(m_A, m_B)$, we have $u(t) \leq (1 + m_A^2 + m_B^2)$. Therefore, by  \cite[Corollary~6.3]{hale2009ordinary}, $\|Y(T-t)\|_F \leq u(T-t)$, and the result follows.  
\end{proof}

\begin{proposition}\label{prop:exist_unique_baseline}
    Let $T \leq \mathcal{T}(m_A, m_B)$ as in \eqref{eq:T_bound}.  Then the coupled Riccati system \eqref{eq:Riccati_A}--\eqref{eq:Riccati_B} for $\theta^A$, $\theta^B$ 
    admits a unique solution over $[0,T]$.
\end{proposition}

\begin{proof}
    Using the same definitions as in Proposition~\ref{prop:theta_bound}, let $Y(t) = \theta(T-t)$ so that $Y$ solves the initial value problem $\dot{Y} = -F(Y) + Q$, $Y(0) = 0$.  
    The right-hand side $G(Y) := -F(Y) + Q$ is quadratic in $Y$, and is therefore locally Lipschitz on any bounded subset of $\R^{4\times 4}$.
    By the Picard--Lindel\"of theorem, there exists a unique solution of $Y$ on $[0,T^\ast)$, where $T^\ast$ is maximal.
    Suppose that $T^* \leq T$; then $\|Y(t)\|_F \to \infty$ as $t \to T^*$. 
    However, on $[0, T^\ast)$ where the solution of $Y(t)$ exists, Proposition~\ref{prop:theta_bound} applies, yielding $\|Y(t)\|_F \leq 1 + m_A^2 + m_B^2$ for all $t \in [0,T^\ast - \epsilon]$, contradicting blowup at $T^*$. 
\end{proof}

\begin{remark}
    The bound \eqref{eq:T_bound} and the condition of Proposition~\ref{prop:exist_unique_thetaAF} below are explicit and checkable; they are sufficient but far from necessary, arising from a conservative ODE comparison argument.  When experiments use a time horizon beyond these bounds, we numerically verify well-posedness of $\theta^A$, $\theta^B$, and $\theta^{AF}$.
\end{remark}

 \section{Game under Partial Information}\label{sec:partial_info}

In the partial-information setting, the state history and the form of the state dynamics and cost functionals are fully known to both players, as are $q_i$, $r_i$, and $\sigma_i$.  However, player $A$ knows only $m_A$ and not the parameter $m_B$ in player $B$'s cost, while player $B$ knows only $m_B$ and not $m_A$.  The opponent's control processes and driving Brownian motions are not directly observed.

To model this information structure, let $M_A \in [m_A^-, m_A^+]$ and $M_B \in [m_B^-, m_B^+]$ be random variables describing the players' beliefs about $m_A$ and $m_B$. Player $A$ knows $M_A = m_A$ and assigns a prior $\pi^B$ to $M_B$, while player $B$ knows $M_B = m_B$ and assigns a prior $\pi^A$ to $M_A$. The priors represent the players' subjective beliefs about the opponent's hidden parameter and are treated as exogenous inputs to the model. In applications, such beliefs may be formed from prior observations, domain knowledge, or previous interactions; we do not assume that they are correctly specified. The bounded supports ensure that the ODE bounds hold uniformly over all admissible realizations of $M_A$ and $M_B$ (cf. Proposition~\ref{prop:theta_bound}).

Since the full-information NE controls \eqref{eq:baseline_NE_control_A}--\eqref{eq:baseline_NE_control_B} depend on both parameters, they are generally not implementable in the partial-information setting.  We therefore introduce implementable controls, which depend only on each player's available information, and derive their explicit forms. We define \emph{implementable controls} by taking expectations of the full-information NE controls under each player's prior:
\begin{align*}
    \overline{u}^{A,*}(t,\bm{x}; m_A) &:= \E^{\pi^B}\!\big[u^{A,*}(t,\bm{x}; m_A, M_B)\big],\\
    \widetilde{u}^{B,*}(t,\bm{x}; m_B) &:= \E^{\pi^A}\!\big[u^{B,*}(t,\bm{x}; M_A, m_B)\big].
\end{align*}
The resulting controls are therefore
\begin{align*}
\overline{u}^{A,\ast}(t, \bm{x}; m_A) &= -\tfrac{1}{r_A}\big(x^A \overline{\theta}_{11}^A(t;m_A) + x^B \overline{\theta}_{12}^A(t;m_A)\big), \\
    \widetilde{u}^{B,\ast}(t, \bm{x}; m_B) &= -\tfrac{1}{r_B}\big(x^A \widetilde{\theta}_{12}^B(t;m_B) + x^B \widetilde{\theta}_{22}^B(t;m_B)\big),
\end{align*}
where $\overline{\theta}^i$, $\widetilde{\theta}^i$, $i\in \{A,B\}$, are defined as
\begin{equation}\label{eq:implementa_thetai}
\begin{aligned}
\overline{\theta}^{A}(t; m_A) := \E^{\pi^B}\big[\theta^{A}(t; m_A, M_B)\big], \\
\widetilde{\theta}^{B}(t; m_B) := \E^{\pi^A}\big[\theta^{B}(t; M_A, m_B)\big],
\end{aligned}
\end{equation}
and \(\theta^A(t; m_A,M_B)\), \(\theta^B(t; M_A,m_B)\) are the matrix-valued solution of \eqref{eq:Riccati_A}--\eqref{eq:Riccati_B} corresponding to the parameter pair \((m_A,M_B)\) and \((M_A,m_B)\) evaluated at time \(t\).  Note that $\overline{\theta}^A$ and $\widetilde{\theta}^B$ are well defined by Proposition~\ref{prop:theta_bound} and the compactness of $[m_A^-, m_A^+]$ and $[m_B^-, m_B^+]$.

\begin{remark}\label{rem:interp_imp_cntrl} 
The construction of the prior-averaged implementable controls $\overline{u}^{A,\ast}$, $\widetilde{u}^{B,\ast}$ is a modeling choice, motivated by three considerations.  First, for each $(t,\bm{x})$, $\overline{u}^{A,\ast}$ is the $L^2(\pi^B)$-projection of the (non-implementable) full-information NE policy onto the space of controls that are measurable with respect to player $A$'s information.  Second, this form preserves semi-analytic tractability, since the control remains linear in $\bm{x}$ with $\overline{\theta}^A$ and $\widetilde{\theta}^B$ computable with  Monte Carlo and standard numerical methods.  Third, the construction is consistent.  Since $m_B \mapsto \theta^A(t;m_A,m_B)$ is continuous as well as
bounded by Proposition~\ref{prop:theta_bound}, weak convergence
$\pi^B \Rightarrow \delta_{m_B}$ yields
$\overline{u}^{A,\ast}\to u^{A,\ast}$; symmetrically,
$\widetilde{u}^{B,\ast}\to u^{B,\ast}$ as $\pi^A \Rightarrow \delta_{m_A}$. 

Other implementable constructions are possible, such as certainty-equivalent controls obtained by substituting a point estimate of the unknown parameter, or a fully belief-dependent strategy obtained by augmenting the state with posterior beliefs.
\end{remark}

Under the implementable controls, the state dynamics are
\begin{equation}
\label{eq:baseline_dynamics}    
\begin{aligned}
    dX_t^A &= \overline{u}^{A,\ast}(t, \bm X_t; m_A) dt + \sigma_A dW_t^A, \\
    dX_t^B &= \widetilde{u}^{B,\ast}(t, \bm X_t; m_B) dt + \sigma_B dW_t^B.  
\end{aligned}
\end{equation}
Equivalently, 
\begin{equation} \label{eq:strong_impl_dyn}  
    d\bm{X}_t = \begin{bmatrix}  -\frac{\overline{\theta}_{11}^A(t; m_A)}{r_A} & -\frac{\overline{\theta}_{12}^A(t; m_A)}{r_A} \\ -\frac{\widetilde{\theta}_{12}^B(t; m_B)}{r_B} & -\frac{\widetilde{\theta}_{22}^B(t; m_B)}{r_B} \end{bmatrix} \bm{X}_t \, dt + \Sigma \, d\bm{W}_t,
\end{equation}
with $\Sigma := \diag\{\sigma_A, \sigma_B\}$. Since the drift is linear in $\bm X$, with coefficients bounded on $[0,T]$, the system admits a unique strong solution.



\section{Alignment Faking  and Information Revelation}\label{sec:AF}
We now formulate the alignment faking (AF) problem for player $A$. This choice is without loss of generality, and the same construction could be instead carried out for player $B$. The goal is for player $A$ to choose a control that extracts information about player $B$'s hidden parameter $m_B$, while also remaining close to the baseline implementable control $\overline{u}^{A,*}$ so as to reduce the risk that player $B$ detects the deviation from the baseline play.  The AF problem for player $A$ is a single-player \emph{control} problem, since player $B$ does not deviate from its baseline implementable control $\widetilde{u}^{B,\ast}$.

\begin{assumption}[Truncated Gaussian Priors]\label{assu:truncated_gaussian_form}
    For the rest of the paper, we assume that both players model the opponent's unknown parameter by truncated Gaussians $
        \pi^A = \mathcal{N}(\mu_A, \rho_A^2; m_A^-, m_A^+)$, $\pi^B = \mathcal{N}\big(\mu_B, \rho_B^2; m_B^-, m_B^+\big)$, 
    where $\mathcal{N}(\mu, \rho^2; a, b)$ denotes a truncated Gaussian with density
    \begin{equation}\label{eq:truncated_Gaussian_density}
        f(m; \mu, \rho, a, b) = \frac{\phi\big(\tfrac{m-\mu}{\rho}\big)}{\rho\big(\Phi(\tfrac{b-\mu}{\rho}) - \Phi (\tfrac{a-\mu}{\rho})\big)}\mathbf{1}_{[a, b]}(m),
    \end{equation}
    and $\phi$ and $\Phi$ represent the density and cumulative distribution functions of a standard  Gaussian $\mathcal{N}(0,1)$.
\end{assumption}

Under Assumption~\ref{assu:truncated_gaussian_form}, player \(A\) knows \(m_A\) and the prior parameters \((\mu_B,\rho_B)\) of \(\pi^B\), but does not know player \(B\)'s realized parameter \(m_B\) nor player \(B\)'s prior parameters \((\mu_A,\rho_A)\) on $m_A$. We collect the quantities unknown to player $A$ relevant to player \(B\)'s control in $\gamma := (m_B,\mu_A,\rho_A)$.  
The choice of a truncated Gaussian to model the priors was made for three reasons: first, the score functions \eqref{eq:score_functions} admit closed forms; second, compact support is required for Proposition~\ref{prop:theta_bound}; third, it is a two-parameter family, providing interpretable axes for the misspecification studies in Section~\ref{sec:numerical_results}.

Let \(v_t\) denote player \(A\)'s AF control, and let player \(B\) use the implementable control \(\widetilde u^{B,*}\). The joint dynamics are
\begin{equation}\label{eq:dynamics_AF}
d\bm X_t
=
\bm b_\gamma\,dt
+
\Sigma\,d\bm W_t,
\quad
\bm b_\gamma
:=
\begin{bmatrix}
v_t\\
\widetilde u^{B,*}(t,\bm X_t;m_B)
\end{bmatrix}.
\end{equation}
The objective here balances closeness to the baseline control (i.e., minimizes $(v_t - \overline{u}^{A,*}_t)^2$) with information acquisition about $m_B$ (cf. \eqref{eq:Schur}). To ensure $v_t$ is implementable, we impose the following assumption.

\begin{assumption}\label{assu:AF}
The AF control \(v_t\) does not explicitly depend on the hidden quantities $\gamma = (m_B, \mu_A, \rho_A)$. Thus player \(A\) designs \(v_t\) using only its available information, namely the realized parameter \(m_A\) and the prior \(\pi^B\).
\end{assumption}

We  focus on implementable  Markovian AF controls for player \(A\), namely controls of the form $
v_t = v(t,\bm X_t; m_A,\pi^B)$, with
$\E\!\big[\int_0^T |v_t|^2\,dt\big] < \infty$.




\subsection{Likelihood Ratio and Fisher Information (FI)}

To quantify the information generated by player $A$'s control about $\gamma$, we derive its Fisher information matrix $I(\gamma)$. 

We begin with the log-likelihood of $\gamma$ based on the observed path $\{\bm{X}_t\}_{0 \leq t \leq T}$. Let \(\Prob_0\) be a reference measure under which \(\bm X\) is driftless. By \cite[Theorem 1.12]{kutoyants2013statistical}, 
 the log-likelihood is
\begin{equation*}
    \ell(\gamma) = \int_0^T \bm{b}_{\gamma}\transpose (\Sigma\Sigma\transpose)^{-1} d\bm{X}_t - \tfrac{1}{2}\int_0^T \bm{b}_{\gamma}\transpose (\Sigma\Sigma\transpose)^{-1} \bm{b}_{\gamma} \, dt.
\end{equation*}
Differentiating with respect to $\gamma_j$ and using \eqref{eq:dynamics_AF}, we obtain $\partial_{\gamma_j} \ell(\gamma) = \int_0^T (\partial_{\gamma_j} \bm{b}_{\gamma})\transpose (\Sigma\Sigma\transpose)^{-1} \Sigma \, d\bm{W}_t$.
Hence the Fisher information matrix $I(\gamma)$ has entries
\begin{align*}
    I(\gamma)_{jk} &= \E\!\Big[\int_0^T (\partial_{\gamma_j} \bm{b}_{\gamma})\transpose (\Sigma\Sigma\transpose)^{-1} (\partial_{\gamma_k} \bm{b}_{\gamma}) \, dt\Big] \notag \\
    &= \tfrac{1}{\sigma_B^2} \E\!\Big[\int_0^T \partial_{\gamma_j} \widetilde{u}^{B,*}_t \cdot \partial_{\gamma_k} \widetilde{u}^{B,*}_t \, dt\Big],
\end{align*}
where the second equality uses Assumption~\ref{assu:AF}, as $\partial_{\gamma_j} v = 0$.

The derivatives $\partial_{\gamma} \widetilde{u}^{B,*}$ are computed by differentiating the functions \(\widetilde\theta^B\) defined in \eqref{eq:implementa_thetai}. A direct calculation gives
\begin{align}
    \partial_{\mu_A} \widetilde{\theta}^B(t; m_B) &= \E^{\pi^A}[\theta^B(t;M_A, m_B) s_{\mu_A}(M_A)], \nonumber\\
    \partial_{\rho_A} \widetilde{\theta}^B(t; m_B) &= \E^{\pi^A}[\theta^B(t;M_A, m_B) s_{\rho_A}(M_A)], \label{eq:thetaBunderprior}\\
    \partial_{m_B} \widetilde{\theta}^B(t; m_B) &= \E^{\pi^A}[\partial_{m_B} \theta^B(t; M_A, m_B)], \nonumber
\end{align}
where $M_A$ has density \eqref{eq:truncated_Gaussian_density}, $\partial_{m_B}\theta^B$ is obtained by differentiating the Riccati system \eqref{eq:Riccati_A}--\eqref{eq:Riccati_B} with respect to $m_B$, yielding linear ODEs with terminal conditions $\partial_{m_B}\theta^i(T) = 0$, and
\begin{equation}\label{eq:score_functions}
\begin{aligned}
    s_{\mu_A}(m) &= \frac{m-\mu_A}{\rho_A^2} + \frac{\phi\big(\tfrac{m_A^+-\mu_A}{\rho_A}\big) - \phi\big(\frac{m_A^--\mu_A}{\rho_A}\big)}{\rho_A\big(\Phi\big(\tfrac{m_A^+-\mu_A}{\rho_A}\big) - \Phi\big(\tfrac{m_A^--\mu_A}{\rho_A}\big)\big)},\\
    s_{\rho_A}(m) &= \tfrac{(m-\mu_A)^2}{\rho_A^3} -\tfrac{1}{\rho_A} \\
    & \quad + \tfrac{\big(\tfrac{m_A^+ -\mu_A}{\rho_A}\big)\phi\big(\tfrac{m_A^+ -\mu_A}{\rho_A}\big) - \big(\tfrac{m_A^- -\mu_A}{\rho_A}\big)\phi\big(\tfrac{m_A^- -\mu_A}{\rho_A}\big)}{\rho_A\big(\Phi\big(\tfrac{m_A^+ -\mu_A}{\rho_A}\big) - \Phi\big(\tfrac{m_A^- -\mu_A}{\rho_A}\big)\big)}.
\end{aligned}
\end{equation}

\begin{remark}[FI and Information Revelation]\label{rem:AF_obj}Under regularity assumptions, if \(\widehat M_B^{(n)}\) is a maximum likelihood estimator (MLE) of $m_B$ based on $n$ independent repetitions of the game, 
    then \(\widehat M_B^{(n)}\) is consistent and asymptotically efficient:
        $\sqrt{n}\big(\widehat{M}_B^{(n)} - m_B \big) \to \mathcal{N}\big(0,[I(\gamma)^{-1}]_{m_B,m_B}\big)$, see  \cite[Theorem~2.8]{kutoyants2013statistical} or  \cite[Theorem 5.1]{lehmann1998theory}.      
    
\end{remark}

Although we do not explicitly model a repeated-game estimation procedure, we use the asymptotic variance \([I(\gamma)^{-1}]_{m_B,m_B}\) as a surrogate for the information-gathering component of player \(A\)'s AF objective. To characterize this quantity, we use the Schur complement and write
 \begin{multline}\label{eq:Schur}
    [I(\gamma)^{-1}]_{m_B,m_B} = \bigl(I_{m_B,m_B} \\ - I_{m_B,(\mu_A,\rho_A)} (I_{(\mu_A,\rho_A),(\mu_A,\rho_A)})^{-1} I_{(\mu_A,\rho_A),m_B}\bigr)^{-1},
\end{multline}
with $I_{(\mu_A,\rho_A),(\mu_A,\rho_A)}$ invertible and $I(\gamma)$ partitioned as
\begin{equation*}
    I(\gamma) = \begin{bmatrix}
        I_{m_B, m_B} & I_{m_B, (\mu_A, \rho_A)} \\
        I_{(\mu_A, \rho_A), m_B} & I_{(\mu_A, \rho_A), (\mu_A, \rho_A)}
    \end{bmatrix}.
\end{equation*}
Moreover, we observe that
\begin{multline}\label{eq:variational}
I_{m_B,m_B}
-
I_{m_B,(\mu_A,\rho_A)}
(I_{(\mu_A,\rho_A),(\mu_A,\rho_A)})^{-1}
I_{(\mu_A,\rho_A),m_B}\\
=
\min_{\bm z\in\mathbb R^2}
\big\{
I_{m_B,m_B}
+
2\bm z^\top I_{(\mu_A,\rho_A),m_B}
\\+
\bm z^\top I_{(\mu_A,\rho_A),(\mu_A,\rho_A)}\bm z
\big\}.
\end{multline}
This representation introduces an auxiliary variable \(\bm z\) and avoids explicit products of Fisher-information entries, which is useful when constructing a quadratic AF objective.

\subsection{Proxy Alignment Faking Objective}

The map $v \mapsto [I(\gamma)^{-1}]_{m_B,m_B}$ is the natural objective for player $A$'s information acquisition, since it is the asymptotic variance of the MLE of $m_B$ (Remark~\ref{rem:AF_obj}).  However, $[I(\gamma)^{-1}]_{m_B,m_B}$ depends on parameters $(m_B,\mu_A, \rho_A)$, which are unknown to player $A$. We therefore construct a proxy objective using only the quantities known to player \(A\), namely \((m_A,\mu_B,\rho_B)\), yielding the max-min problem~\eqref{eq:saddle}.

To motivate the proxy construction of $[I(\gamma)^{-1}]$, suppose the truncation interval \([m_A^-,m_A^+]\) is wide relative to \(\rho_A\), so that \(\pi^A\) is close to the Gaussian \(\mathcal N(\mu_A,\rho_A^2)\). Writing \(M_A \approx \mu_A+\rho_A Z\) with \(Z\sim\mathcal N(0,1)\), the sensitivities in \eqref{eq:thetaBunderprior} may then be approximated by
 \begin{align*}
    &\E^{Z}\left[\partial_{m_A}\theta^B(t; \mu_A + \rho_A Z, m_B)\right], \\
    &\E^{Z}\left[Z\partial_{m_A}\theta^B(t; \mu_A + \rho_A Z, m_B)\right], \\
    &\E^{Z}\left[\partial_{m_B}\theta^B(t; \mu_A + \rho_A Z, m_B)\right].
 \end{align*} 
Since player $A$ does not know $\mu_A$ or $\rho_A$, none of these expected sensitivities can be computed by player $A$.  We therefore construct implementable \emph{proxy} sensitivities using only quantities available to player $A$. For the first argument of $\theta^B$, player $A$ replaces player $B$'s unknown belief about $m_A$ by the realized value, which is known exactly to player $A$. For the second argument, the unknown realization $m_B$ is represented using player $A$'s own prior $M_B\sim\pi_B$, approximated here by $\mu_B+\rho_B Z$. Retaining the functional form of the true expected sensitivities with these substitutions, we define:
\begin{align*}
    \partial_{\mu_A} \overline{\widetilde{\theta}}^B(t;m_A) &:= \E^{Z}\left[\partial_{m_A}\theta^B(t; m_A, \mu_B + \rho_B Z)\right], \\
    \partial_{\rho_A} \overline{\widetilde{\theta}}^B(t; m_A) &:= \E^{Z}\left[Z\partial_{m_A}\theta^B(t; m_A, \mu_B + \rho_B Z)\right], \\
    \partial_{m_B} \overline{\widetilde{\theta}}^B(t; m_A) &:= \E^{\pi^B}[\partial_{m_B}\theta^B(t; m_A, M_B)],
\end{align*}
where the last definition is for notational consistency.  

Player $A$ then uses $\partial_{\gamma} \overline{\widetilde{\theta}}^B$ as a \emph{proxy} for $\partial_{\gamma} \widetilde{\theta}^B$. This yields the proxy Fisher information matrix $\overline{I}$, whose $jk$-th entry is
\begin{multline*}
    \tfrac{1}{\sigma_B^2 r_B^2}\, \E\!\Big[\int_0^T \!\big(X_t^A \partial_{\gamma_j} \overline{\widetilde{\theta}}_{12}^B + X_t^B \partial_{\gamma_j} \overline{\widetilde{\theta}}_{22}^B\big)\\
    \cdot \big(X_t^A \partial_{\gamma_k} \overline{\widetilde{\theta}}_{12}^B + X_t^B \partial_{\gamma_k} \overline{\widetilde{\theta}}_{22}^B\big) dt\Big]. 
\end{multline*}
As with $I(\gamma)$, we partition $\overline{I}$ as 
\begin{equation*}
    \overline{I} = \begin{bmatrix}
        \overline{I}_{m_B, m_B} & \overline{I}_{m_B, (\mu_A, \rho_A)} \\
        \overline{I}_{(\mu_A, \rho_A), m_B} & \overline{I}_{(\mu_A, \rho_A), (\mu_A, \rho_A)}
    \end{bmatrix}.
\end{equation*}

Using the proxy $\overline{I}$ and the variational identity \eqref{eq:variational}, we define the proxy AF cost functional $J^{AF}$ as 
\begin{multline*}
    J^{AF}[v, \bm{z}] := \E\!\Big[\int_0^T  q^{AF}(v_t - \overline{u}^{A,*}_t)^2 +  r^{AF} v_t^2 \, dt\Big] \\
    -  \lambda^{AF} \big(\overline{I}_{m_B,m_B} + 2\bm{z}\transpose \overline{I}_{(\mu_A,\rho_A),m_B} + \bm{z}\transpose \overline{I}_{(\mu_A,\rho_A),(\mu_A,\rho_A)} \bm{z}\big),
\end{multline*}
where $ q^{AF},  r^{AF} > 0$ penalize deviation from the implementable baseline $\overline{u}^{A,*}_t$ and control effort, respectively, while $ \lambda^{AF} > 0$ rewards information acquisition.  
We then consider the saddle-point problem
\begin{equation}\label{eq:saddle}
    \max_{\bm{z} \in \R^2}\; \min_{v}\; J^{AF}[v, \bm{z}].
\end{equation}
Importantly, player $A$ optimizes $J^{AF}$ under dynamics 
\begin{equation}\label{eq:dynamics_AF_playerA}
    d\bm{X}_t = \begin{bmatrix} v_t \\ \overline{u}^{B,*}_t\end{bmatrix} dt + \Sigma\, d\bm{W}_t,
\end{equation}
where $\widetilde u^{B, \ast}$ in \eqref{eq:dynamics_AF} is replaced by $ \overline{u}^{B, \ast}_t := \overline u^{B, \ast} (t,\bm X_t; m_A)$ with $\overline u^{B, \ast} (t,\bm x; m_A) = \E^{\pi^B}\big[u^{B,\ast}(t, \bm x; m_A, M_B)\big]$
since player $A$ is unable to compute $B$'s true control $\widetilde{u}^{B,\ast}$. 



\subsection{Solution via Quadratic Ansatz and Gradient Descent}

We now solve the saddle-point problem \eqref{eq:saddle}--\eqref{eq:dynamics_AF_playerA}. For $\bm{z}= [ z_1, z_2 ]\transpose$, define the terms 
\begin{align*}
    g_{12}(t,\bm{z}) &:= \tfrac{1}{r_B}\Big(z_1 \partial_{\mu_A} \overline{\widetilde{\theta}}_{12}^B + z_2 \partial_{\rho_A} \overline{\widetilde{\theta}}_{12}^B + \partial_{m_B} \overline{\widetilde{\theta}}_{12}^B\Big),\\
    g_{22}(t,\bm{z}) &:= \tfrac{1}{r_B}\Big(z_1 \partial_{\mu_A} \overline{\widetilde{\theta}}_{22}^B + z_2 \partial_{\rho_A} \overline{\widetilde{\theta}}_{22}^B + \partial_{m_B} \overline{\widetilde{\theta}}_{22}^B\Big).
\end{align*}

For fixed $\bm{z} \in \R^2$, let $V^{AF}(t,\bm x;\bm z)$ be the value function associated with $J^{AF}$ starting at $\bm X_t = \bm x$. Using the HJB approach and a quadratic ansatz 
\[
V^{AF}(t,\bm x;\bm z)
=
\bm x\transpose \theta^{AF}(t; \bm z)\bm x
+
\psi^{AF}(t; \bm z),\]
with $\theta^{AF}(t)\in\mathbb S^{2\times 2}$, the minimizer over $v$ is given by
\begin{multline}\label{eq:v_ast_z}
    v^\ast(t, \bm x; \bm{z}) = -\tfrac{1}{ q^{AF} +  r^{AF}}\Big( \tfrac{ q^{AF}}{r_A}\big(x^A \overline{\theta}_{11}^A(t) + x^B \overline{\theta}_{12}^A(t)\big)\\
    + x^A \theta_{11}^{AF}(t; \bm z) + x^B \theta_{12}^{AF}(t; \bm z)\Big),
\end{multline}
where $\theta^{AF}(t; \bm z)$ solves the Riccati ODE 
\begin{multline}\label{eq:theta_AF_ODE}
    \dot{\theta}^{AF} = \tfrac{1}{ q^{AF}+ r^{AF}}\theta^{AF} R_A \theta^{AF}\\ 
    \quad+ \tfrac{ q^{AF}}{( q^{AF}+ r^{AF})r_A}\big(\theta^{AF} R_A \overline{\theta}^A + \overline{\theta}^A R_A \theta^{AF}\big) \\
    \quad+ \theta^{AF} R_B \overline{\theta}^B + \overline{\theta}^B R_B \theta^{AF} - \tfrac{ q^{AF}  r^{AF}}{( q^{AF}+  r^{AF})r_A^2}\overline{\theta}^A R_A \overline{\theta}^A \\
    \quad - \tfrac{ \lambda^{AF}}{\sigma_B^2} \begin{bmatrix} g_{12}(t,\bm{z}) \\ g_{22}(t,\bm{z}) \end{bmatrix}\begin{bmatrix}  g_{12}(t,\bm{z}) \\ g_{22}(t,\bm{z}) \end{bmatrix}\transpose,
\end{multline}
with terminal condition $\theta^{AF}(T) = 0$.  

Computing the saddle-point solution of \eqref{eq:saddle} requires solving \eqref{eq:theta_AF_ODE}  for each candidate $\bm{z}$, simulating the corresponding controlled state paths, and optimizing over $\bm{z}$.  
We therefore propose an iterative scheme in Algorithm~\ref{alg:opt}, which alternates between exact minimization in $v$ through the Riccati ODEs and one gradient-ascent step in $\bm{z}$ using Monte Carlo estimates and automatic differentiation.  To ensure existence of $\theta^{AF}$ (cf. Proposition~\ref{prop:exist_unique_thetaAF}), each gradient step of $\bm{z}$ is projected onto $B_r:=\{\bm{z}:\|\bm{z}\|\leq r\}$ via the Euclidean projection $\Pi_{B_r}$, for some $r>0$.


\begin{algorithm}[htp]
    \caption{Optimization of $J^{AF}$}\label{alg:opt}
    \KwIn{Initial value $\bm{z}^{(0)} \in \R^2$, projection radius $r>0$, convergence tolerance $\epsilon > 0$, learning rate $\alpha > 0$, sample size $N_{\text{MC}} \in \mathbb{N}$.}
    Initialize $k \gets 0$, $J^{AF}_\mathrm{prev} \gets -\infty$, $J^{AF}_\mathrm{curr} \gets \infty$\;
    \While{$\big|J^{AF}_\mathrm{curr} - J^{AF}_\mathrm{prev} \big| > \epsilon$}{
        Set $J^{AF}_\mathrm{prev} \gets J^{AF}_\mathrm{curr}$\;
        Compute $\theta^{AF}(t; \bm{z}^{(k)})$\;
        Compute $v^\ast(t, \bm x; \bm{z}^{(k)})$ according to \eqref{eq:v_ast_z}\;
        Sample $N_{\text{MC}}$ paths of $X^A$ under $v^{\ast}(t, \bm X_t; \bm{z}^{(k)})$ and $X^B$ under $\overline{u}^{B,\ast}(t, \bm X_t; m_A)$\;
        Compute $J^{AF}_\mathrm{curr} \approx J^{AF}[v^\ast(\cdot; \bm{z}^{(k)}), \bm{z}^{(k)}]$ with Monte Carlo\;
        Compute $\nabla_{\bm{z}} J^{AF}_\mathrm{curr} \approx \nabla_{\bm{z}} J^{AF}[v^\ast(\cdot; \bm{z}^{(k)}), \bm{z}^{(k)}]$ with automatic differentiation\;
        Set $\bm{z}^{(k+1)} \gets \Pi_{B_r}\big(\bm{z}^{(k)} + \alpha \nabla_{\bm{z}} J^{AF}_\mathrm{curr}\big)$\;
    }
    \Return{$\big(\bm{z}^{\ast}, v^{\ast}(t, \bm x; \bm{z}^{\ast})\big)$}.
\end{algorithm}

\begin{proposition}
\label{prop:exist_unique_thetaAF}
    Let $c_A := \max\{|m_A^-|, |m_A^+|\}$, $c_B := \max\{|m_B^-|, |m_B^+|\}$, and assume $\|\bm{z}\| \leq r$ for some $r>0$.  Define the constants 
    \begin{align*}
        c_\theta &= \max\Big\{\sup\limits_{[0,\mathcal{T}(c_A,c_B)]}\|\overline{\theta}^A\|_F, \sup\limits_{[0,\mathcal{T}(c_A,c_B)]}\|\overline{\theta}^B\|_F\Big\},\\
        c_g &= \sup\limits_{[0,\mathcal{T}(c_A,c_B)], \|\bm{z}\| \leq r}\|[ g_{12}, g_{22}]\transpose \|_F.  
    \end{align*}
    If $T < \min\big\{\mathcal{T}(c_A, c_B), \pi\big(2\sqrt{(d_1 + \tfrac{1}{2})\big(\tfrac{d_2^2}{2} + d_3\big)}\big)^{-1}\big\}$,
    where $d_1$, $d_2$, and $d_3$ are defined as 
        $d_1 := \tfrac{1}{( q^{AF}+ r^{AF})r_A}$,  $d_2 := \tfrac{2  q^{AF} c_\theta}{( q^{AF} +  r^{AF})r_A^2} + \tfrac{2 c_\theta}{r_B}$, and  $d_3 := \tfrac{ q^{AF}  r^{AF} c_\theta^2}{( q^{AF} +  r^{AF})r_A^3} + \tfrac{ \lambda^{AF} c_g^2}{\sigma_B^2}$,
    then equation \eqref{eq:theta_AF_ODE} admits a unique solution on \([0,T]\).
\end{proposition}


\begin{proof}
    We argue as in Propositions~\ref{prop:theta_bound} and \ref{prop:exist_unique_baseline}. 
    First, we bound $\theta^{AF}$ when a solution exists.  
    Let $Y(t)= \theta^{AF}(T-t)$. 
    From \eqref{eq:theta_AF_ODE}, applying the triangle inequality and submultiplicativity of $\|\cdot\|_F$ produces $\|\dot{Y}\|_F \leq d_1\|Y\|_F^2 + d_2\|Y\|_F + d_3$.
    By Young's inequality for products, $d_2\|Y\|_F \leq \tfrac{1}{2}\|Y\|_F^2 + \tfrac{d_2^2}{2}$,
    so that $\|\dot{Y}\|_F \leq \big(d_1 + \tfrac{1}{2}\big)\|Y\|_F^2 + \tfrac{d_2^2}{2} + d_3. $
    
    Set $\hat{d_1} := d_1 + \tfrac{1}{2}$, $\hat{d_3} := \tfrac{d_2^2}{2} + d_3$ and consider the scalar comparison ODE $
        \dot{u} = \hat{d_1} u^2 + \hat{d_3}, \; u(0) = 0.$   As in Proposition~\ref{prop:theta_bound}, if $T < \tfrac{\pi}{2\sqrt{\hat{d_1}\hat{d_3}}}$ then $u(t)$ remains bounded on $[0,T]$.
    By \cite[Corollary~6.3]{hale2009ordinary}, $\|Y(t)\|_F \leq u(t)$, so any local solution of \eqref{eq:theta_AF_ODE} is bounded on \([0,T]\).
    
    Lastly, as in Proposition~\ref{prop:exist_unique_baseline}, 
since the right-hand side of \eqref{eq:theta_AF_ODE} is locally Lipschitz in \(\theta^{AF}\), the Picard--Lindel\"of theorem gives a unique local solution near \(T\). The boundedness established above rules out finite-time blow-up, so the local solution extends uniquely to all of \([0,T]\).   
\end{proof}

\subsection{Detection of Alignment Faking}\label{sec:detection}
We now consider a heuristic detection procedure from player \(B\)'s perspective. Using the prior \(\pi^A\), player \(B\) forms its own implementable prediction of player \(A\)'s baseline control:
\begin{equation*}
    \widetilde{u}^{A, \ast}_t := \E^{\pi^A}\big[ u^{A,\ast}(t, \bm{X}_t; M_A, m_B)\big].
\end{equation*}
Player $B$ then considers the residual process
\[d\nu_t := dX_t^A - \widetilde{u}^{A,\ast}_t\,dt, \]
which measures the discrepancy between the observed state increment and the increment predicted by player \(B\)'s model of player \(A\)'s baseline behavior.

Under the non-AF dynamics \eqref{eq:baseline_dynamics}, $d\nu_t = \big(\overline{u}_t^{A,\ast} - \widetilde{u}^{A,\ast}_t\big) dt + \sigma_AdW_t^A$.
Since both \(\overline u^{A,*}\) and \(\widetilde u^{A,*}\) are linear in states, the drift of $\nu_t$ is linear in the state:
\begin{align*}
    [\overline{u}^{A,\ast} - \widetilde{u}^{A,\ast}](t, \bm x) &= -\tfrac{1}{r_A}\Big(x^A\big(\overline{\theta}_{11}^A - \widetilde{\theta}_{11}^A\big) + x^B\big(\overline{\theta}_{12}^A - \widetilde{\theta}_{12}^A\big)\Big) \notag \\
    &=: \delta_{11}^A(t) x^A + \delta_{12}^A(t)x^B,
\end{align*}
where  $\delta_{11}^A$ and $\delta_{12}^A$ are deterministic functions depending on the discrepancy between the two players' priors.

If player \(A\) instead uses the AF control \(v^*(t, \bm x; \bm z^*)\) from \eqref{eq:saddle}, then $d\nu_t = \big(v^{\ast}(t, \bm X_t; \bm{z}^\ast) - \widetilde{u}_t^{A,\ast}\big) dt + \sigma_A dW_t^A$,
and the drift becomes
\begin{align*}
    v^\ast(t, \bm x; &\bm{z}^\ast) - \widetilde{u}^{A,\ast}(t, \bm x; m_B) \\
    &= -\tfrac{1}{ q^{AF}+ r^{AF}}\Big(\tfrac{ q^{AF}}{r_A}\big(x^A \overline{\theta}_{11}^A + x^B \overline{\theta}_{12}^A\big) \notag\\
    &\qquad + x^A \theta_{11}^{AF} + x^B \theta_{12}^{AF}\Big) + \tfrac{1}{r_A}\big(x^A \widetilde{\theta}_{11}^A + x^B \widetilde{\theta}_{12}^A\big) \notag \\
    &=: \eta_{1}^{AF}(t; \bm z^\ast) x^A + \eta_{2}^{AF}(t; \bm z^\ast) x^B, \quad\text{where} \\
    \eta_1^{AF}&(t; \bm z^\ast) := \tfrac{1}{r_A}\widetilde{\theta}_{11}^A - \tfrac{1}{ q^{AF}+ r^{AF}}\Big(\tfrac{ q^{AF}}{r_A}\overline{\theta}_{11}^A + \theta_{11}^{AF}\Big), \\
    \eta_2^{AF}&(t; \bm z^\ast) := \tfrac{1}{r_A}\widetilde{\theta}_{12}^A - \tfrac{1}{ q^{AF}+ r^{AF}}\Big(\tfrac{ q^{AF}}{r_A}\overline{\theta}_{12}^A + \theta_{12}^{AF}\Big).
\end{align*}
Thus, under both non-AF and AF objectives, the drift of $\nu_t$ is linear in the state, but with different coefficients.

For numerical detection, we discretize the residual as
\[
    \Delta\nu_{t_k}
    :=
    X_{t_{k+1}}^A-X_{t_k}^A-\widetilde u^{A,*}(t_k,\bm X_{t_k}; m_B)\Delta t.
\]
If the game were played repeatedly, we could then perform, for each time step \(t_k\), a least-squares regression of \(\Delta\nu_{t_k}\) on \(X_{t_k}^A\) and \(X_{t_k}^B\): $
    \Delta\nu_{t_k}
    \approx
    \alpha_1(t_k)X_{t_k}^A+\alpha_2(t_k)X_{t_k}^B+\epsilon_k$,
where \(\epsilon_k\sim\mathcal N(0,\sigma_A^2\Delta t)\).

When the priors \(\pi^A\) and \(\pi^B\) are reasonably accurate, our numerical experiments in Section~\ref{subsec:traj_info_detec} indicate that the magnitudes of \(\alpha_1\) and \(\alpha_2\) tend to be larger under AF than under baseline play. This motivates the detection statistic
\begin{equation}\label{eq:detection_mag}
    D^{AF} := \max\big\{ \max\limits_{t_k} \big|\alpha_1(t_k)\big|, \max\limits_{t_k} \big|\alpha_2(t_k)\big| \big\}.
\end{equation}
Heuristically, larger values of \(D^{AF}\) indicate a higher likelihood that player \(A\) is using the AF control rather than the baseline implementable control.

Since both hypotheses render the drift of $\nu_t$ linear in the state with known Gaussian noise, a likelihood-ratio test between the two drift specifications is a natural formalization, which we leave to future work.  In our numerical experiments (see Figure~\ref{fig:traj}), \(D^{AF}\) exhibits a consistent separation between AF and baseline controls across the tested parameter regimes. This detection procedure requires player \(B\) to observe multiple independent realizations of the game, which is in line with the discussion in Remark~\ref{rem:AF_obj}. As expected, when the priors \(\pi^A\) and \(\pi^B\) are substantially misspecified, or when \(q^{AF}/\lambda^{AF}\) is large, the separation between AF and baseline detection scores decreases. This reflects the tradeoff between information gain and detectability. 



\section{Numerical Illustrations}\label{sec:numerical_results}

In this section, we present numerical experiments illustrating the main findings of the paper. We compare the asymptotic variance of \(m_B\), given by \([I(\gamma)^{-1}]_{m_B,m_B}\), with its proxy counterpart \([\overline{I}^{-1}]_{m_B,m_B}\) under both AF and baseline controls. Our experiments focus on three aspects: the relationship between the accuracy of the priors, information gain, and detectability; the effect of mean misspecification in \(\pi^A\) and \(\pi^B\) on information quality; and the role of the AF intensity \( \lambda^{AF}\) when the prior means are correctly specified. Unless otherwise stated, all experiments use the following parameter configuration:
\begin{gather*}
    T = 1.0,\quad m_A = m_B = 1.0, \quad \mu_A = \mu_B = 1.0, \\
    q_A = r_A = q_B = r_B = r^{AF} = 1.0, \quad \sigma_A = \sigma_B = 0.1. 
\end{gather*}
Controlled SDEs are simulated by the Euler--Maruyama scheme on a uniform grid of \(N=100\) time steps over \([0,T]\).  All expectations are estimated via Monte Carlo with $N_{\text{MC}}=50{,}000$ independent samples.  For all experiments, Algorithm~\ref{alg:opt} uses $\bm{z}^{(0)} = \bm{0}, \; \epsilon = 10^{-4}, \; \alpha = 0.05.$ We observe stable numerical convergence across all experiments, with termination occurring within 70 iterations in every case. In addition, motivated by Proposition~\ref{prop:exist_unique_thetaAF}, we monitor \(\bm z\) throughout the iterations and observe that \(\|\bm z\|<2\) in all experiments, so with $r=2$ in Algorithm~\ref{alg:opt} the projection $\Pi_{B_r}$ never activates.




\subsection{Trajectories, Information, and Detectability}\label{subsec:traj_info_detec}

Figure~\ref{fig:traj} provides a combined view of player $A$'s AF control problem and information acquisition.  We examine the case that the mean of $\pi^A$ is correctly specified, while the mean of $\pi^B$ ranges from correct to misspecified.

\begin{figure}[htbp]
    \centering
    \includegraphics[width=1.0\linewidth]{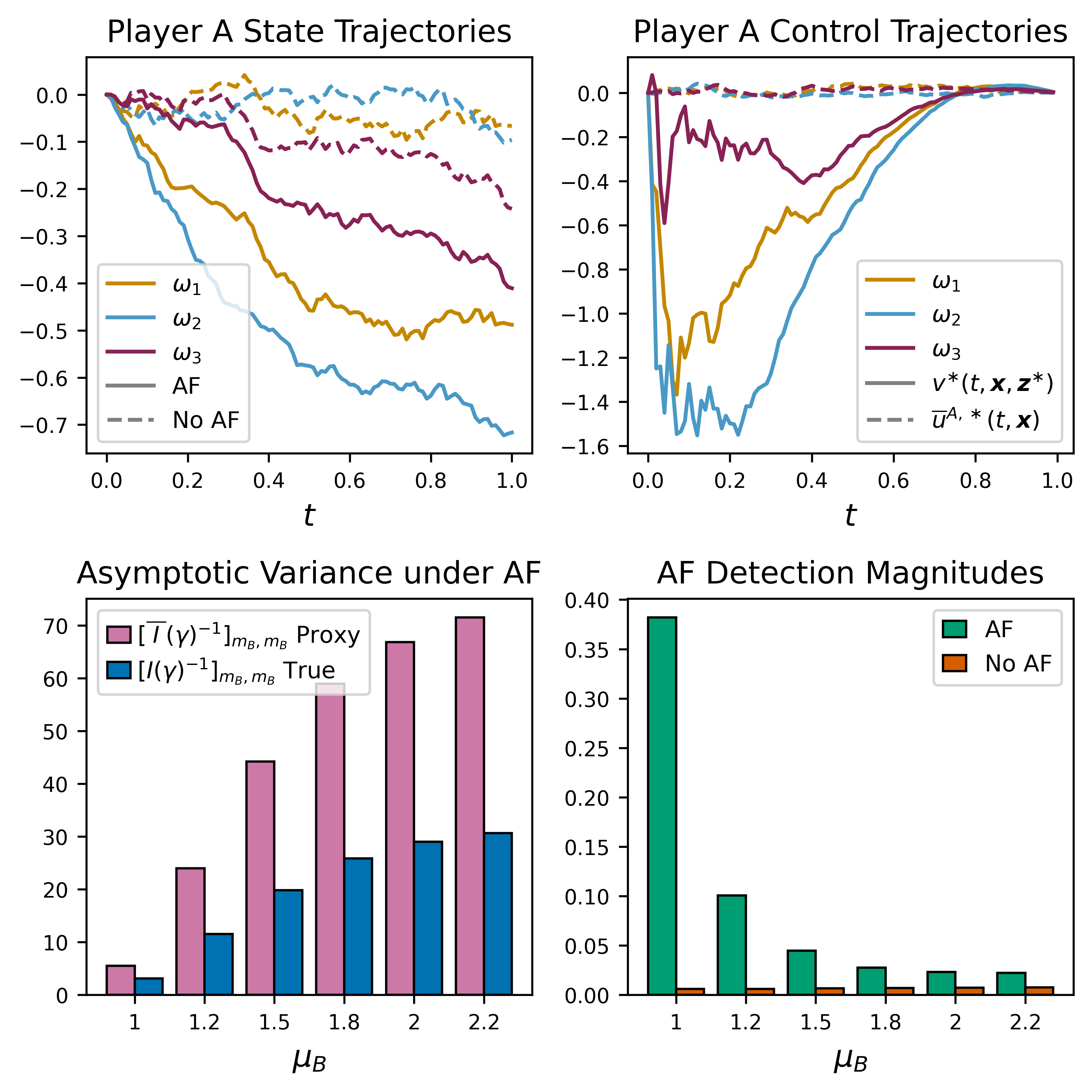}
    \caption{
    Combined view of AF behavior with fixed $\mu_A = m_A= 1.0$ and varying $\mu_B \in \{1.0, 1.2, 1.5, 1.8, 2.0, 2.2\}$.
    Top panels show state (left) and control (right) trajectories in the case that $\mu_B = m_B = 1.0$.
    Bottom panels show asymptotic variance (left) and regression-based  detectability $D^{AF}$ (right), both computed via Monte Carlo.
    Parameters: $ q^{AF} = 5.0$, $ \lambda^{AF} = 2.5$, and $\rho_A = \rho_B = 0.1$.
    }
    \label{fig:traj}
\end{figure}

The top panels show three sampled state and control trajectories of player $A$, corresponding to three independent realizations $\omega_1, \omega_2, \omega_3$,  in the case $\mu_B=1.0$. The AF control $v^\ast$ exhibits noticeably larger inputs than the baseline control $\overline{u}^{A,\ast}$, reflecting the alignment faker's (player $A$'s) active exploration to extract information about $m_B$. 

The bottom-left panel compares the asymptotic variance $[I(\gamma)^{-1}]_{m_B, m_B}$ and its proxy under the AF control as $\mu_B$ varies.  When $\mu_B = m_B$, both the proxy and true variance are small, indicating effective information extraction.  As $\mu_B$ increases, both variances grow, but the proxy variance increases faster than the true one, showing that the proxy systematically overestimates uncertainty under misspecification of \(\pi^B\). Thus, player \(A\)'s self-assessment of information quality becomes increasingly unreliable as its model becomes more inaccurate. The bottom-right panel examines detectability via the regression-based procedure described in Section~\ref{sec:detection}.  For each value of $\mu_B$, we report $D^{AF}$ defined in \eqref{eq:detection_mag} under both the AF and baseline controls.  When $\mu_B$ is close to the true value $m_B$, the AF detection magnitude is substantially larger than the baseline, indicating that the AF behavior is readily detectable.  As $\mu_B$ increases, the AF detection score moves closer to the baseline level, since a more misspecified $\pi^B$ yields AF controls closer to the baseline.  This highlights a tradeoff: an accurate prior $\pi^B$ improves information gain but also increases detectability.

\subsection{Effect of Mean Misspecification on Information} 
We next examine the effect of mean misspecification in both players' priors in Figure~\ref{fig:measure_info}, by varying both $\mu_A$ and $\mu_B$. Under the AF control (solid lines), the true asymptotic variance $[I(\gamma)^{-1}]_{m_B,m_B}$ increases with $\mu_B$, showing that player $A$'s more misspecified model degrades information quality. By contrast, varying $\mu_A$ produces only modest separation among the curves, suggesting that $\pi^B$ is the main driver of the true information gained by the alignment faker, while \(\pi^A\) plays a secondary role. 

Under the baseline control (dashed lines), the true asymptotic variance is nearly constant across all $\mu_A$ and $\mu_B$, and remains substantially higher than the AF variance when $\mu_B$ is close to the true value $m_B$.  This behavior of the baseline curves is explained by the lack of an information-seeking term in the baseline objective; the misspecification of each player's prior only dilutes the effectiveness of their controls, with little effect on the Fisher information.  In contrast, the AF control rewards information acquisition, and the resulting control lowers the true asymptotic variance, producing the gap from the baseline.  We also note that the proxy asymptotic variance, although not shown, increases as $\pi^B$ becomes more misspecified. However, since the proxy $\overline{I}^{-1}$ depends only on $\pi^B$ by design, $\pi^A$ has no effect.  

\begin{figure}[htbp]
    \centering
    \includegraphics[width=0.9\linewidth]{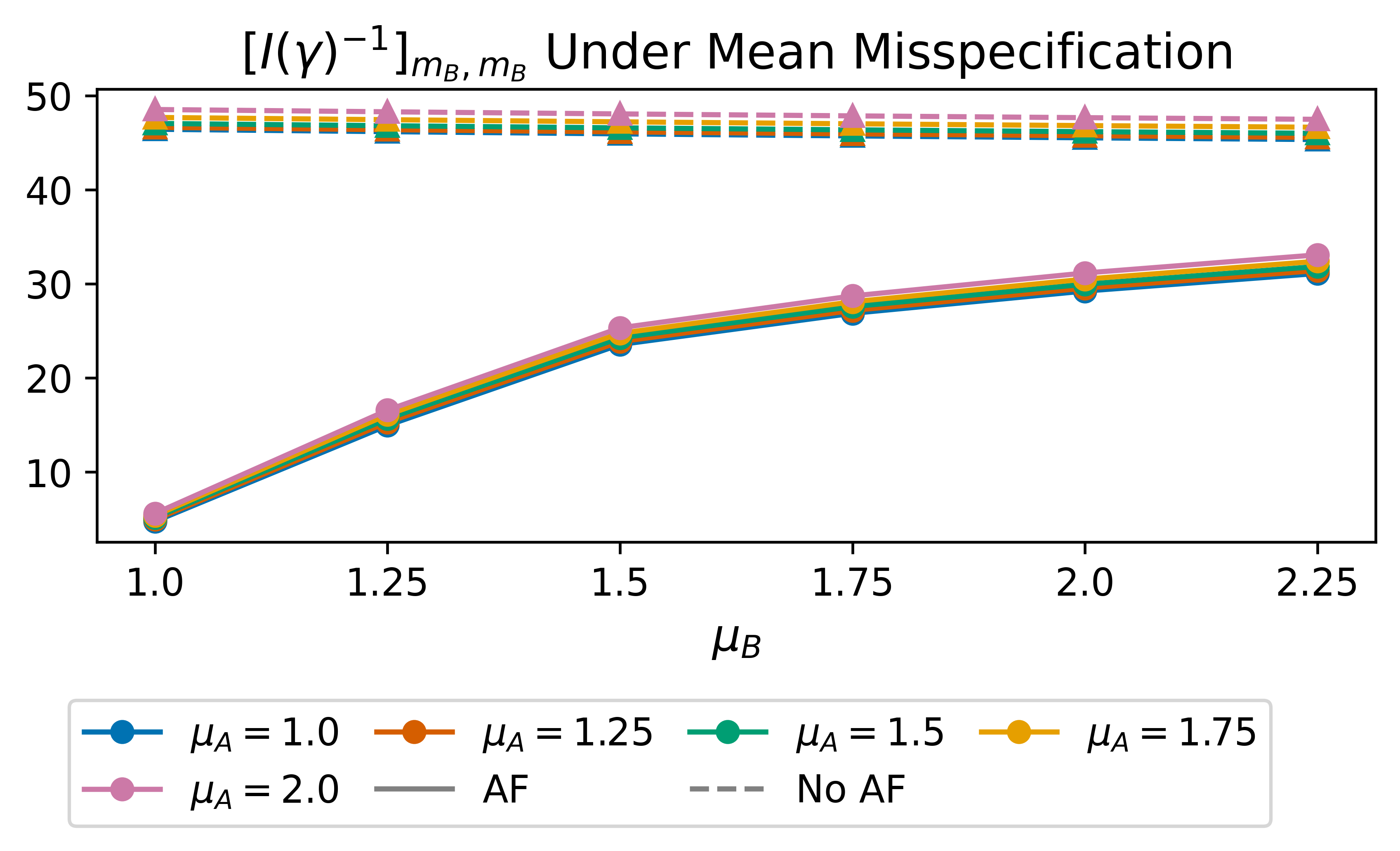}
    \caption{
    True asymptotic variance $[I(\gamma)^{-1}]_{m_B,m_B}$ for $\mu_A \in \{1.0, 1.25, 1.5, 1.75, 2.0\}$ and $\mu_B \in \{1.0, 1.25, 1.5, 1.75, 2.0, 2.25\}$ under both AF (solid) and no AF (dashed) gameplay.
    Parameters: $q^{AF} = 5.0$, $\lambda^{AF} = 2.5$, and $\rho_A = \rho_B = 0.1$.
    }
    \label{fig:measure_info}
\end{figure}

Overall, these results show that a well-specified prior $\pi^B$ allows player $A$ to achieve a substantial information gain relative to the baseline control, while misspecification in \(\pi^A\) has only a limited effect on the true information extracted.


\subsection{Role of $ \lambda^{AF}$ under AF Control}

Finally, Figure~\ref{fig:lam_info} examines how the AF intensity $ \lambda^{AF}$ affects true and proxy asymptotic variance when the means of $\pi^A$ and $\pi^B$ are fixed and the standard deviations vary.  

\begin{figure}[htbp]
    \centering
    \includegraphics[width=0.9\linewidth]{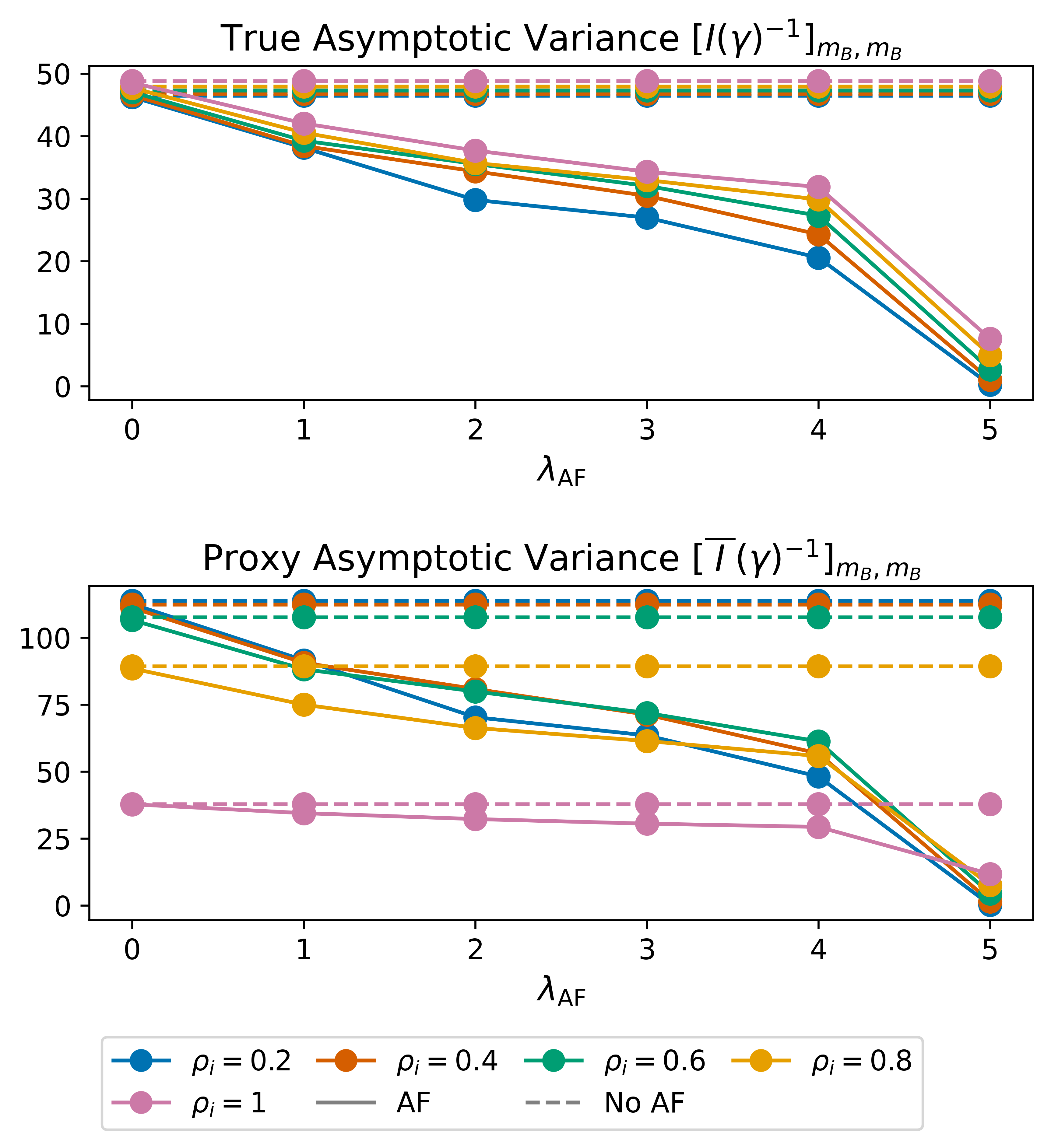}
    \caption{True asymptotic variance \([I(\gamma)^{-1}]_{m_B,m_B}\) (top) and proxy asymptotic variance \([\overline{I}^{-1}]_{m_B,m_B}\) (bottom) versus \(\lambda^{AF}\in \{0,1,2,3,4,5\}\), for standard deviations \(\rho_A=\rho_B\in \{10^{-7},0.1,0.2,\ldots,1.0\}\) and under both AF (solid) and no AF (dashed) gameplay. Each point is computed by Monte Carlo. Parameters: \(q^{AF}=10.0\), \(\mu_A=\mu_B=1.0\).
    }
    \label{fig:lam_info}
\end{figure}

In the top panel, as $ \lambda^{AF}$ increases, the true asymptotic variance $[I(\gamma)^{-1}]_{m_B, m_B}$ under the AF control (solid lines) decreases, with smaller standard deviations $\rho_A = \rho_B$ yielding more pronounced reductions.  The baseline variance (dashed lines) remains approximately constant across all $ \lambda^{AF}$.   The bottom panel shows that the proxy asymptotic variance $[\overline{I}^{-1}]_{m_B,m_B}$ exhibits a similar decreasing trend as $ \lambda^{AF}$ increases.  Interestingly, we observe that the proxy consistently overestimates the true variance at small $\rho_i$ and underestimates it at larger $\rho_i$, suggesting that proxy information can be systematically biased. 

Altogether, this confirms that information gain is driven by the AF control rather than the baseline dynamics, and that 
\(\lambda^{AF}\) serves as an effective parameter for tuning the amount of information extracted.

\section{Conclusions and Future Work}

We studied a two-player stochastic differential game under partial information to investigate information revelation when each player reasons through a proxy model of the opponent. Each player knows its own coupling parameter and models the opponent's through a prior. Starting from the full-information Nash equilibrium, we derived implementable baseline controls by averaging the equilibrium controls under the players' priors and characterized them through coupled Riccati systems. We then formulated the alignment faking problem as a tractable saddle-point control problem that remains close to the implementable baseline while optimizing a proxy Fisher information objective. Numerical results showed that alignment faking can significantly increase information gain, but often at the cost of greater detectability, and that the proxy Fisher information may be systematically misleading under model misspecification. Future work includes extending the framework beyond the linear-quadratic setting, allowing both players to act strategically on information revelation, incorporating dynamic learning and belief updating, and studying richer detection rules and their effect on the tradeoff between information gain and camouflage.



\bibliographystyle{IEEEtran}
\bibliography{references}

\begin{thebibliography}{10}
\providecommand{\url}[1]{#1}
\csname url@samestyle\endcsname
\providecommand{\newblock}{\relax}
\providecommand{\bibinfo}[2]{#2}
\providecommand{\BIBentrySTDinterwordspacing}{\spaceskip=0pt\relax}
\providecommand{\BIBentryALTinterwordstretchfactor}{4}
\providecommand{\BIBentryALTinterwordspacing}{\spaceskip=\fontdimen2\font plus
\BIBentryALTinterwordstretchfactor\fontdimen3\font minus
  \fontdimen4\font\relax}
\providecommand{\BIBforeignlanguage}[2]{{%
\expandafter\ifx\csname l@#1\endcsname\relax
\typeout{** WARNING: IEEEtran.bst: No hyphenation pattern has been}%
\typeout{** loaded for the language `#1'. Using the pattern for}%
\typeout{** the default language instead.}%
\else
\language=\csname l@#1\endcsname
\fi
#2}}
\providecommand{\BIBdecl}{\relax}
\BIBdecl

\bibitem{townsend1983forecasting}
R.~M. Townsend, ``Forecasting the forecasts of others,'' \emph{Journal of
  Political Economy}, vol.~91, no.~4, pp. 546--588, 1983.

\bibitem{greenblatt2024alignment}
R.~Greenblatt, C.~Denison, B.~Wright, F.~Roger, M.~MacDiarmid, S.~Marks,
  J.~Treutlein, T.~Belonax, J.~Chen, D.~Duvenaud \emph{et~al.}, ``Alignment
  faking in large language models,'' \emph{arXiv:2412.14093}.

\bibitem{hubinger2019risks}
E.~Hubinger, C.~Van~Merwijk, V.~Mikulik, J.~Skalse, and S.~Garrabrant, ``Risks
  from learned optimization in advanced machine learning systems,''
  \emph{arXiv:1906.01820}, 2019.

\bibitem{hu2025strategic}
R.~Hu, D.~Ralston, X.~Yang, and H.~Zhou, ``Strategic inference in {S}tackelberg
  games: {O}ptimal control for revealing adversary intent,''
  \emph{arXiv:2510.05641}, 2025.

\bibitem{bacsar1998dynamic}
T.~Ba{\c{s}}ar and G.~J. Olsder, \emph{Dynamic Noncooperative Game
  Theory}.\hskip 1em plus 0.5em minus 0.4em\relax SIAM, 1998.

\bibitem{carmona2016lectures}
R.~Carmona, \emph{Lectures on {BSDE}s, Stochastic Control, and Stochastic
  Differential Games with Financial Applications}.\hskip 1em plus 0.5em minus
  0.4em\relax SIAM, 2016.

\bibitem{zhou2025adversarial}
H.~Zhou, D.~Ralston, X.~Yang, and R.~Hu, ``Adversarial decision-making in
  partially observable multi-agent systems: {A} sequential hypothesis testing
  approach,'' \emph{arXiv:2509.03727}, 2025.

\bibitem{zhou2025integrating}
------, ``Integrating sequential hypothesis testing into adversarial games: {A}
  {S}un {Z}i-inspired framework,'' in \emph{2025 IEEE 64th Conference on
  Decision and Control (CDC)}.\hskip 1em plus 0.5em minus 0.4em\relax IEEE,
  2025, pp. 4540--4546.

\bibitem{kim2026stealthy}
Y.~Kim, H.~Zhou, A.~Benvenuti, R.~Hu, and M.~Hale, ``Deception in
  linear-quadratic control,'' \emph{arXiv:2604.00227}, 2026.

\bibitem{kim2024defining}
Y.~Kim, A.~Benvenuti, B.~Chen, M.~Karabag, A.~Kulkarni, N.~D. Bastian,
  U.~Topcu, and M.~Hale, ``Defining and measuring deception in sequential
  decision systems: {A}pplication to network defense,'' in \emph{2024 IEEE
  Military Communications Conference}, 2024, pp. 1--6.

\bibitem{morris2002social}
S.~Morris and H.~S. Shin, ``Social value of public information,''
  \emph{American Economic Review}, vol.~92, no.~5, pp. 1521--1534, 2002.

\bibitem{zhang2020game}
T.~Zhang, L.~Huang, J.~Pawlick, and Q.~Zhu, ``Game-theoretic analysis of cyber
  deception: {E}vidence-based strategies and dynamic risk mitigation,''
  \emph{Modeling and Design of Secure Internet of Things}, pp. 27--58, 2020.

\bibitem{chang2014linear}
D.~Chang and H.~Xiao, ``Linear quadratic nonzero sum differential games with
  asymmetric information,'' \emph{Mathematical Problems in Engineering}, vol.
  2014, no.~1, p. 262314, 2014.

\bibitem{cardaliaguet2012games}
P.~Cardaliaguet and C.~Rainer, ``Games with incomplete information in
  continuous time and for continuous types,'' \emph{Dynamic Games and
  Applications}, vol.~2, no.~2, pp. 206--227, 2012.

\bibitem{harsanyi1967games}
J.~C. Harsanyi, ``Games with incomplete information played by {B}ayesian
  players, {P}art {I}: {T}he basic model,'' \emph{Management science}, vol.~14,
  no.~3, pp. 159--182, 1967.

\bibitem{molloy2017inverse}
T.~L. Molloy, J.~J. Ford, and T.~Perez, ``Inverse noncooperative dynamic
  games,'' \emph{IFAC-PapersOnLine}, vol.~50, no.~1, pp. 11\,788--11\,793,
  2017.

\bibitem{peters2021inferring}
L.~Peters, D.~Fridovich-Keil, V.~Rubies-Royo, C.~J. Tomlin, and C.~Stachniss,
  ``Inferring objectives in continuous dynamic games from noise-corrupted
  partial state observations,'' \emph{arXiv:2106.03611}, 2021.

\bibitem{ward2025active}
W.~Ward, Y.~Yu, J.~Levy, N.~Mehr, D.~Fridovich-Keil, and U.~Topcu, ``Active
  inverse learning in {S}tackelberg trajectory games,'' in \emph{2025 American
  Control Conference (ACC)}.\hskip 1em plus 0.5em minus 0.4em\relax IEEE, 2025,
  pp. 1547--1553.

\bibitem{abou2012matrix}
H.~Abou-Kandil, G.~Freiling, V.~Ionescu, and G.~Jank, \emph{Matrix {R}iccati
  Equations in Control and Systems Theory}.\hskip 1em plus 0.5em minus
  0.4em\relax Birkh{\"a}user, 2012.

\bibitem{papavassilopoulos1979existence}
G.~Papavassilopoulos, J.~Medanic, and J.~Cruz~Jr, ``On the existence of {N}ash
  strategies and solutions to coupled {R}iccati equations in linear-quadratic
  games,'' \emph{Journal of Optimization Theory and Applications}, vol.~28,
  no.~1, pp. 49--76, 1979.

\bibitem{hale2009ordinary}
J.~K. Hale, \emph{Ordinary Differential Equations}.\hskip 1em plus 0.5em minus
  0.4em\relax Courier Corporation, 2009.

\bibitem{kutoyants2013statistical}
Y.~A. Kutoyants, \emph{Statistical Inference for Ergodic Diffusion
  Processes}.\hskip 1em plus 0.5em minus 0.4em\relax Springer Science \&
  Business Media, 2013.

\bibitem{lehmann1998theory}
E.~L. Lehmann and G.~Casella, \emph{Theory of Point Estimation}.\hskip 1em plus
  0.5em minus 0.4em\relax Springer, 1998.

\end{thebibliography}

\end{document}